\documentclass[leqno, 12pt]{amsart}

\usepackage[utf8]{inputenc}
\usepackage{lmodern}
\usepackage[english]{babel}
\usepackage{amstext}
\usepackage{amsthm}
\usepackage{amsmath}
\usepackage{amssymb}
\usepackage{graphicx}
\usepackage{mathtools}
\usepackage{float}
\usepackage[svgnames]{xcolor}
\definecolor{color1}{RGB}{27,158,119}
\definecolor{color2}{RGB}{217,95,2}
\definecolor{color3}{RGB}{117,112,179}
\definecolor{color4}{RGB}{231,41,138}
\usepackage[margin=1in]{geometry}
\usepackage{dsfont}
\usepackage{fancyvrb}
\usepackage[unicode=true,pdfusetitle,
bookmarks=true,bookmarksnumbered=false,bookmarksopen=false,
breaklinks=false,pdfborder={0 0 1},backref=false,colorlinks=true]
{hyperref}
\usepackage{enumitem}
\usepackage{tikz}
\usepackage{pgfplots}
\pgfplotsset{compat=1.18}

\usepackage{placeins}

\hypersetup{
	pdftitle={Ground States for the Defocusing Nonlinear Schrödinger Equation on Non-Compact Metric Graphs},
	pdfauthor={Elio Durand-Simonnet and Boris Shakarov},
	pdfsubject={},
	pdfkeywords={}, 
	colorlinks=true,
	linkcolor=DarkBlue  ,          %
	citecolor=DarkRed,        %
	filecolor=DarkMagenta,      %
	urlcolor=DarkGreen,           %
}

\newtheorem{theorem}{Theorem}[section]
\newtheorem{lem}[theorem]{Lemma}
\newtheorem{prop}[theorem]{Proposition}
\newtheorem{proposition}[theorem]{Proposition}
\newtheorem{cor}[theorem]{Corollary}
\newtheorem{corollary}[theorem]{Corollary}
\newtheorem{remark}[theorem]{Remark}
\theoremstyle{definition}

\def\R{\mathbb R}

\def\N{\mathbb N}

\def\eps{\varepsilon}
\def\H{\operatorname{H}}
\def\G{\mathcal{G}}

\DeclarePairedDelimiterX{\dual}[2]{\langle}{\rangle}{#1, #2}

\title[NLS Ground states on Quantum Graphs]{Ground States for the Defocusing Nonlinear Schrödinger Equation on Non-Compact Metric Graphs}

\author[E. Durand-Simonnet]{Elio Durand-Simonnet}
\address{Institut de Mathématiques de Toulouse ; UMR5219, Université de Toulouse ; CNRS, UPS IMT, F-31062 Toulouse Cedex 9 (France)}
\email{elio.durand\_simonnet@math.univ-toulouse.fr}

\author[B. Shakarov]{Boris Shakarov}
\address{Institut de Mathématiques de Toulouse ; UMR5219, Université de Toulouse ; CNRS, UPS IMT, F-31062 Toulouse Cedex 9 (France)}
\email{boris.shakarov@math.univ-toulouse.fr}
\thanks{The work of E.D.S. and B. S. is 
	partially supported by ANR project CIMI ANR-11-LABX-0040 and the ANR project NQG ANR-23-CE40-0005. E.D.S. is 
	partially supported by ANR project MINT ANR-18-EURE-0023}
\date{\today}


\date{\today}

\begin{document}
		
\begin{abstract}
   We investigate the existence and stability of ground states for the defocusing nonlinear Schrödinger equation on non-compact metric graphs. We establish a sharp criterion for the existence of action ground states in terms of the spectral properties of the underlying Hamiltonian: ground states exist if and only if the bottom of the spectrum is negative and the frequency lies within a suitable range. We further explore the relation between action and energy ground states, showing that while every action minimizer yields an energy minimizer, the converse fails in general. In particular, we prove that energy ground states may not exist for arbitrary masses. This discrepancy is illustrated through explicit examples on star graphs with $\delta$ and $\delta'$-type vertex conditions: in the mass-subcritical case, we exhibit a large interval of masses for which no energy minimizer exists, whereas in the supercritical regime, energy ground states exist for all masses.
\end{abstract}

\maketitle

\section{Introduction}

Quantum graphs, which model wave propagation in networks, have received considerable attention in both mathematical physics and nonlinear analysis. In particular, the nonlinear Schrödinger equation (NLS) on quantum graphs provides a fundamental framework for studying nonlinear wave phenomena in branched structures, with applications in diverse areas such as optics \cite{Ag00, LeAs98} and condensed matter physics \cite{BiBaAl01, BrGiMa04, BuCaRaSo01}.

In this work, we investigate the defocusing nonlinear Schrödinger equation on metric graphs, where the underlying differential operator \( \H \) is the Laplacian endowed with general vertex conditions, as described in \cite[Chapter~1]{BeKu13} and detailed in Section~\ref{secPreliminaries}. More specifically, we study the defocusing NLS on a \emph{non-compact} metric graph \( \G \), meaning a graph composed of a finite set of edges \( \mathcal{E} \) and vertices \( \mathcal{V} \), with at least one edge of infinite length.

The operator \( \H \) acts as the Laplacian along each edge, i.e.,
\[
    \operatorname{H} \psi = (-\psi_e'')_{e \in \mathcal E},
\]
and is equipped with suitable vertex conditions at each vertex \( v \in \mathcal{V} \), drawn from the Shapiro–Lopatinski framework \cite{LiMa72}. We introduce the associated quadratic form
\begin{equation*}
    Q_{\H}(\psi) = \langle \operatorname{H} \psi, \psi \rangle_{L^2(\mathcal G)},
\end{equation*}
with domain \( D(Q_{\H}) \subset H^1(\G) \), depending on the vertex conditions. The precise definitions of \( \H \), as well as of the relevant Lebesgue and Sobolev spaces on \( \G \), are provided in Section~\ref{secPreliminaries}.

We consider the defocusing nonlinear Schrödinger equation
\begin{equation} \label{eqTimeDepNLS}
    \begin{cases}
        i \partial_t u = \operatorname{H} u + |u|^{p-1} u, \\
        u(t = 0) = u_0 \in D(Q_{\H}),
    \end{cases}
\end{equation}
for \( u : \mathbb{R} \times \mathcal{G} \to \mathbb{C} \) and exponent \( p > 1 \). The initial datum \( u_0 \) belongs to the energy space \( D(Q_{\H}) \), ensuring local well-posedness of \eqref{eqTimeDepNLS} via standard contraction arguments \cite{Ca03}. Global well-posedness follows from the conservation of two quantities, the \emph{energy} and the \emph{mass}, defined for all \( \psi \in D(Q_{\H}) \) by
\begin{equation}
    \label{eqEnMass}
    E(\psi) = \frac{1}{2} Q_{\H}(\psi) + \frac{1}{p + 1} \| \psi \|_{L^{p + 1} (\G)}^{p + 1}, 
    \qquad
    M(\psi) = \frac{1}{2} \| \psi \|_{L^2(\G)}^2.
\end{equation}

We will primarily focus on the associated stationary problem, obtained through the standard ansatz
\begin{equation*}  
    u(t) = e^{i \omega t} \psi,
\end{equation*}  
where \( \omega \in \mathbb{R} \) is the frequency parameter, and \( \psi \in D(Q_{\H}) \) is time-independent. Substituting this ansatz into \eqref{eqTimeDepNLS} reduces the problem to the stationary defocusing nonlinear Schrödinger equation
\begin{equation} \label{eqStationaryNLS}  
    \operatorname{H} \psi + \omega \psi + |\psi|^{p - 1} \psi = 0.
\end{equation}  
Equation \eqref{eqStationaryNLS} governs the standing wave solutions of the system.

These solutions can be characterized as critical points of the \textit{action} functional 
\[
    S_{\omega} (\psi) = E(\psi) + \omega M(\psi), \qquad \psi \in D(Q_{\H}),
\]
leading to the variational problem 
\begin{equation} \label{eqMinAction}
    s_{\omega} = \inf \left\{ S_{\omega} (\psi): \psi \in D(Q_{\H})\right\}.
\end{equation}
If a minimizer exists, it satisfies $S'_{\omega} (\psi) = 0$, i.e., equation \eqref{eqStationaryNLS}. We denote the set of all minimizers to \eqref{eqMinAction} by 
\begin{equation}\label{eqActionSet}
    \mathcal{A}_\omega = \{ \psi \in D(Q_{\H}) : S_\omega(\psi) = s_\omega\},
\end{equation}
and refer to them as \emph{action ground states}. 

An alternative approach to \eqref{eqStationaryNLS} is via the \emph{energy minimization problem} at fixed mass $m > 0$:
\begin{equation}\label{eqEnMinProb}
    e_m = \inf \{ E(\psi) : \psi \in D(Q_{\H}), \, M(\psi) = m \}.
\end{equation}
Here the parameter $\omega$ in \eqref{eqStationaryNLS} arises as a Lagrange multiplier. A natural question is whether the infimum $e_m$ is attained, i.e., whether there exists a function \( \psi \in D(Q_{\H}) \) such that $M(\psi) = m$ and $E(\psi) = e_m$. We refer to such minimizers as \textit{energy ground states} and denote the set by
\begin{equation} \label{eqBDefEnGs}
    \mathcal{B}_m =  \{ \psi \in D(Q_{\H}) : E(\psi) = e_m, \ M(\psi) = m \}.
\end{equation}

Our primary objective in this work is to establish a necessary and sufficient condition for the existence of action ground states. Before proceeding with our analysis, we first review the relevant literature on the subject.

Quantum graphs provide a simplified one-dimensional model for studying wave propagation in fiber-like structures. Their effectiveness in approximating higher-dimensional counterparts has been investigated in \cite{Ko00, Ko02}. In \cite{TeTzVi14}, a connection was established between energy/mass minimizers on product spaces and ground states on  $\mathbb{R}^d \times \mathcal M$ where $\mathcal M$ is a compact manifold.
The Gross-Pitaevskii equation on  $\mathbb{R} \times \mathbb{T}$ with a general nonlinearity was recently studied in \cite{deGrSm24, MaMu24}, where it was shown that energy/momentum minimizers become effectively one-dimensional as the torus  $\mathbb{T}$  shrinks to zero. The bifurcation between a one-dimensional line soliton and a fully two-dimensional soliton on the cylinder  $\mathbb{R} \times \mathbb{T}$  was analyzed in \cite{AkBaIbKi24, Ya14, Ya15} More recently in \cite{CoSh24} a line with a  $\delta$-potential was obtained as the limiting case of a fractured strip $\R \times [0,L]$ with Neumann boundary conditions. 

The study of the nonlinear Schrödinger equation on quantum graphs has gained significant attention over the past two decades, particularly in the case of the focusing nonlinearity (where the nonlinear term appears with a minus sign). In this setting, numerous works have investigated the existence of standing waves and their stability. The first contribution can be reconnected to the case of the line with a delta potential, as an example of a star-graph with two edges. In this case, the works \cite{FuJe08, FuOhOz08,GoHoWe04} establish the existence of action ground states and their stability,  and in \cite{CoFuFi08} the full picture of stability and instability is provided for different choices of parameters in the equation. For the energy minimizers, the existence of energy ground states is proven in \cite{AdNoVi13} on the line with a quadratic form which includes those considered here. 

Similar results were later obtained in more general graphs. Notably, the classical concentration compactness principle of Lions \cite{Li85} is slightly modified for the graphs. Specifically, in addition to the three standard behaviors of vanishing, compactness, and dichotomy, an additional possible behavior must be accounted for, the \textit{run-away}, where mass escapes to infinity on one infinite edge, see \cite{AdCaFiNo14}. This gives rise to a general criterion of existence for both energy and action ground states on graphs. Roughly speaking, if the energy (or the action) of the line ground state is superior to one of the possible candidates for the minimization problem, then a ground state exists as the \textit{run-away} behavior can be ruled out. Then vanishing and dichotomy are ruled out due to the convexity of the energy. Due to this criterion, the existence of energy ground states is proven in \cite{AdCaFiNo14,AdCaFiNo14A, AdSeTi15, AdSeTi17, BoCa23, CoDoGaSeTr23}.
For a general recent review on the focusing NLS on quantum graphs, we refer to \cite{KaNoPe22}.

In contrast, the defocusing case has received significantly less attention, despite its mathematical and physical relevance \cite{PiSt16}. The present work aims to contribute to bridging this gap.

In  $\mathbb{R}$, it is well known that for \( \operatorname{H} = -\partial_{xx} \), $H^1(\R)$ solutions to \eqref{eqStationaryNLS} do not exist. However, in bounded domains, solutions to the defocusing stationary equation do exist under various boundary conditions. This follows from the compact Sobolev embedding \( H^1(I) \hookrightarrow L^p(I) \) for any  $p \in [2, \infty]$  and classical variational methods, which will be discussed later in this work (see, for example \cite{GuCoTs17}).

Quantum graphs lie at the interface of these two settings, combining features of both bounded and unbounded domains. This makes them an interesting framework to investigate. To the best of our knowledge, the defocusing stationary NLS on graphs has been investigated in only two works. In \cite{KaOh09}, the problem was studied on the real line with a negative delta potential, where existence and uniqueness (up to a phase shift) of solutions to \eqref{eqStationaryNLS} were established. This result was later extended in \cite{PaGo18} to star graphs. In both cases, the analysis relies crucially on the availability of explicit solutions, made possible by the symmetry of the star-graph and the absence of finite-length edges. In contrast, the present work is carried out in a general setting, where no explicit solutions are available, symmetries cannot be exploited, and the presence of a nontrivial compact core requires a different approach.

Moreover, to the best of our knowledge, this is the first work that investigates the nonlinear Schrödinger equation under such a general class of vertex conditions, even when compared to the extensive literature on the focusing case. Previous studies have almost exclusively concentrated on Kirchhoff or $\delta$-type couplings, leaving the role of more general self-adjoint conditions essentially unexplored.

Given the scarcity of results in the defocusing case and under general vertex conditions, our aim is to advance the theory by considering the NLS on a general non-compact quantum graph endowed with arbitrary vertex conditions, as discussed in Section~\ref{secPreliminaries}. To illustrate the abstract results, we will then analyze in detail two specific examples in Section~\ref{secStarGraphs}, namely star graphs with $\delta$- and $\delta'$-type vertex conditions.

The defocusing case presents several important differences compared to the focusing one. In particular, the minimization of the action functional does not require additional constraints, such as restricting to the Nehari manifold. As a consequence, the search for action minimizers is more straightforward than the identification of energy ground states.

Furthermore, since the vertex conditions act locally, the main difficulty in the construction of an action ground state reduces to ruling out the vanishing scenario through appropriate scaling arguments. In this framework, the existence criterion for action ground states simplifies to the sharp condition
\[
    s_\omega < 0.
\]

Since the only possible negative contribution to the action functional comes from the quadratic form, the existence criterion requires a careful balance between the choice of vertex conditions and the frequency parameter $\omega$. In particular, if we denote by $l_{\H}$ the bottom of the spectrum of $\operatorname{H}$, namely
\begin{equation} \label{eqLGamOm}
    l_{\H} = \inf \left\{ \left\langle \operatorname{H} \psi, \psi \right\rangle_{L^2(\mathcal G)} : \psi \in D(Q_{\H}), \ \| \psi \|_{L^2(\G)} = 1 \right\},
\end{equation}
then a necessary condition for the existence of action minimizers is that $l_{\H} < 0$. 

Our main result shows that this spectral condition is also sufficient, provided the frequency is chosen in the appropriate range. More precisely, as long as $\omega \in (0, -l_{\H})$, the existence of action ground states is guaranteed. Moreover, in the borderline case $\omega = 0$, the same conclusion holds for subcritical nonlinearities $p \in (1,5)$.

\begin{theorem} \label{thmCriterion}
    Let $\G$ be a non-compact metric graph and $\omega \neq 0$. Then an action ground state exists if and only if $l_{\H} < 0$ and $\omega \in (0, -l_{\H})$. In addition, for $\omega = 0$ and $p \in (1,5)$, an action ground state exists under the same  condition $l_{\H} < 0$.
\end{theorem}
In other words, ground states exist precisely when the Hamiltonian $\H$ admits a bound state, and the frequency $\omega$ is chosen within the corresponding spectral gap.

For the case $\omega = 0$ and $p \geq 5$, the situation becomes more delicate: the existence of a minimizer is no longer guaranteed, and a general existence criterion is difficult to formulate, as it depends sensitively on the type and distribution of vertex conditions across the graph. Nevertheless, we show that if a minimizer exists, it must vanish identically on every infinite edge—that is, it necessarily has compact support and is supported only on edges of finite length. This phenomenon is, in fact, a general feature of solutions to \eqref{eqStationaryNLS} in this regime: when $p \geq 5$ and $\omega = 0$, any solution must be entirely supported within the compact core.

\begin{theorem}\label{thmCritCase}
Let $\G$ be a non-compact graph and $l_{\H} < 0$, $\omega = 0$ and $p \geq 5$. Then any non-trivial solution $u\in D(Q_{\H})$ of \eqref{eqStationaryNLS} must vanish identically on every infinite edge of the graph. In particular, $u$ is supported on finite-length edges.
\end{theorem}

Thus, in most cases, the existence of action ground states for the nonlinear model reduces to analyzing the spectral properties of the Hamiltonian operator $\H$.  For star graphs, a criterion for the existence of a negative eigenvalue of the Hamiltonian is provided in \cite[Theorem 3.7]{KoSc06}. In the case of graphs containing at least one compact edge, \cite[Theorem 3.2]{KoSc06} establishes a criterion for determining whether a given scalar  $l \in \mathbb{C}$  is an eigenvalue of \( \operatorname{H} \). 

As a consequence of the conservation of the action, we will also obtain the stability of the set $\mathcal{A}_\omega$ for the equation \eqref{eqTimeDepNLS}. A set $\mathcal A$ is said to be \textit{stable} if for any $\varepsilon > 0$, there exists $\eta > 0$ such that if $u_0 \in D(Q_{\H})$ satisfies 
    $$ 
    \inf_{\psi \in \mathcal{A}} \| u_0 - \psi \|_{H^1(\mathcal G)} < \eta,
    $$ 
    then the corresponding solution $u$ to \eqref{eqTimeDepNLS} satisfies
\[
    \sup_{t \geq 0} \inf_{\psi \in \mathcal{A}} \| u(t, \cdot) - \psi \|_{H^1(\mathcal G)} < \varepsilon.
\] 
Finally, we will also show that the function $\omega \mapsto s_\omega$
is continuous and strictly increasing.
\begin{cor}\label{corStability}
   Let $l_{\H}<0$ and $\omega \in [0, -l_{\operatorname{H}})$ for $p \in (1,5)$, or $\omega \in (0,-l_{\H})$ for $p \geq 5$. Then the function $\omega \mapsto s_\omega$ is continuous and strictly increasing, and the set $\mathcal{A}_\omega$ is stable.
\end{cor}

In contrast to action ground states, the existence of energy ground states for the defocusing NLS is significantly more challenging. Notably, the energy functional is no longer convex for all masses. In some cases, the minimization problem \eqref{eqEnMinProb} satisfies  $e_m > -\infty$, yet no minimizer exists due to the possibility of dichotomy in the concentration-compactness argument. In other words, an energy minimizer may not exist for a given mass  $m > 0$, as its existence depends on the structure of the graph and the parameters in equation \eqref{eqStationaryNLS}. A concrete example illustrating this phenomenon will be presented in Section \ref{secStarGraphs}. 

Another key difference from the focusing case lies in the relationship between action and energy ground states. In the focusing case, energy minimizers typically also minimize the action. However, the conditions under which the reverse holds remain unclear. For further discussion on this, we refer to \cite{DeDoGaSe23, DoSeTi23, JeLu22}.

In our setting, the situation is reversed: every action ground state is also an energy ground state. However, the converse is more difficult to establish, as the relationship between the parameter  $\omega$  and the mass of action ground states remains unclear. As a first step in addressing this issue, we demonstrate that this question is linked to the uniqueness of energy minimizers. Specifically, we prove that if two distinct energy minimizers exist and have two different associated Lagrangian multipliers, then at least one of them cannot be an action minimizer.

\begin{theorem} \label{thmExisEnGs}
Let \( \omega \in [0, - l_{\H}) \) and  $\phi \in \mathcal{A}_\omega$  with  $m = M(\phi)$, then  $\phi \in \mathcal{B}_m$. Moreover, if for  $m > 0$, there exist  $\psi, \phi \in \mathcal{B}_m$ such that  $\psi \neq \phi$  and satisfying \eqref{eqStationaryNLS} with respectively $\omega_\psi, \omega_\phi$ such that $\omega_\psi = \omega_\phi$, then either  $\psi \notin \mathcal{A}_{\omega_\psi}$ or  $\phi \notin \mathcal{A}_{\omega_\phi}$.
\end{theorem}

In Section \ref{secCriterion}, we will prove Theorems \ref{thmCriterion} and \ref{thmExisEnGs}, along with additional properties of ground states. Section \ref{secStarGraphs} then focuses on the specific case of a star graph with  $\delta $ or $\delta'$-vertex conditions. Section \ref{secStarGraphs} serves two main purposes.

First, in this setting, explicit solutions to \eqref{eqStationaryNLS} are available and are unique up to a phase shift. The explicit solution for the $\delta$-vertex conditions was derived in \cite{PaGo18}, while we will compute the solution for the $\delta'$-vertex conditions.

Second, exploiting these explicit solutions, we will demonstrate that for  $p < 5$, there exists a large interval of masses for which no energy ground state exists. Specifically, we will show that there exists a threshold $m_p$ with $m_p = \infty$ for $p \geq 5$ and $m_p < \infty$ for $p\in (1,5)$ such that energy ground states exist only for  $m\in (0,m_p)$. This is shown in Proposition \ref{prpSec4Main}.   

The remainder of this paper is organized as follows. In Section \ref{secPreliminaries}, we introduce the necessary notation and preliminaries, including the spectral properties of the operator $\H$ and the well-posedness of problem \eqref{eqTimeDepNLS}. In Section \ref{secCriterion}, we establish the proof of Theorem \ref{thmCriterion} and Theorem \ref{thmCritCase}, providing a rigorous criterion for the existence of ground states. Finally, in Section \ref{secStarGraphs}, we focus on the specific case of the star graph with $\delta$ and $\delta'$-vertex conditions, which is of particular interest as it allows for the construction of explicit solutions, providing deeper insights into the analytical properties of ground states.

\section{Preliminaries} \label{secPreliminaries}

\subsection{Metric graphs}

In this paper, we consider a metric graph $\mathcal G = (\mathcal E, \mathcal V)$, with $\mathcal E$ the set of edges and $\mathcal V$ the set of vertices.  The sets $\mathcal E$ and $\mathcal V$ will always be supposed of finite cardinality. For $v \in \mathcal V$, we denote by $J_v$ the set of edges adjacent to $v$ and by $d_v$ the cardinal of $J_v$. 
An edge $e \in \mathcal E$ has a length $L_e$ that can be finite or infinite. The coordinate system on the metric graph is largely arbitrary. When the length $L_e$ is finite, the edge  $e \in \mathcal E$ can be identified with a different choice for an interval $I_e \subset \R$, as $I_e = [0, L_e]$.  When the length $L_e$ is infinite, the edge can be identified with any half-line $I_e = [a,\infty)$, where $a \in \R$ can be chosen freely. We highlight that throughout this work, every graph $\G$ considered has at least one infinite edge. 
For later convenience, we define $L$ as the minimum between half of the minimum length between every edge and $1/2$ 
\begin{equation} \label{eqMinLenght}
    L = \frac{1}{2} \min \left( \min_{e \in \mathcal E} L_e, 1 \right).
\end{equation}
This quantity is well defined and $L >0$ since $|\mathcal E| < \infty$. A point $\tilde{x} \in \G$ on the graph can be identified by giving the edge $e \in \mathcal{E}$ and the coordinate $x$ on the edge: $\tilde{x} = (e, x)$. With slight abuse of notation, we will use $x$ for both a point on the graph and a coordinate on one edge when there is no ambiguity.  

A function on the metric graph is defined component-wise, giving its value on any edge. Assigning  $\psi$  on  $\G$  corresponds to specifying its edge components  $\{\psi_e\}_{e \in E}$, where each  $\psi_e$  is a function  $\psi_e: I_e \to \mathbb{C}$. Equivalently, for a point  $\tilde x = (e, x)$  on the graph, we define  $\psi(\tilde x) = \psi_e(x)$. Since a general metric graph  $\G$  is described by a collection of points and intervals, a natural choice of metric structure on  $\G$  is given by the Lebesgue spaces defined on these intervals. In particular, if a function $\psi$ is integrable, then every component $\psi_e$ on the edge is integrable on $I_e$ and
\[
    \int_{\G} \psi (x) dx = \sum_{e \in \mathcal E} \int_{I_e} \psi_e (x) dx.
\]
For $p \geq 2$ and  $k \in \N$, this allows to define the Lebesgue and Sobolev spaces 
\begin{equation*}
    L^p (\G) = \bigoplus_{e \in \mathcal E} L^p(I_e), \quad H^k (\G) = \bigoplus_{e \in \mathcal E} H^k(I_e) 
\end{equation*}
endowed with the norms given by 
\begin{equation*}
    \left\| \psi \right\|_{L^p (\mathcal G)}^p = \sum_{e \in \mathcal E} \left\| \psi_e \right\|_{L^p(I_e)}^p, \quad \left\| \psi \right\|_{H^k (\mathcal G)}^2 = \sum_{e \in \mathcal E} \left\| \psi_e \right\|_{H^k(I_e)}^2,
\end{equation*}
and the scalar product on $L^2(\G)$ as
\[
    \langle \psi, \phi \rangle_{L^2 (\mathcal G)} = \operatorname{Re} \left( \int_{\G} \psi (x) \overline{\phi(x)} dx \right) = \sum_{e \in \mathcal E} \operatorname{Re} \left( \int_{I_e} \psi_e(x) \overline{\phi_e(x)} dx \right).
\]
For $\psi \in H^1(\mathcal G), v \in \mathcal V$ and $e \in J_v$, we denote by $\psi_e(v)$ the value of $\psi_e$ at the vertex $v$ and by $\Psi(v) \in \mathbb C^{d_v}$ the vector given by
\[
    \Psi(v) = (\psi_{e_1} (v), \ldots, \psi_{e_{d_v}} (v)),
\]
where $\{ e_1, \dots, e_{d_v} \} = J_v$. For $q \in [2,\infty]$ and $\psi \in H^1(\mathcal G)$, the Gagliardo–Nirenberg inequality
\begin{equation} \label{eqGN}
    \| \psi \|_{L^q (\mathcal G)} \leq C \| \psi' \|_{L^2 (\mathcal G)}^{\frac{1}{2}-\frac{1}{q}} \| \psi \|_{L^2 (\mathcal G)}^{\frac{1}{2}+\frac{1}{q}}
\end{equation}
holds for some $C > 0$ \cite{AdSeTi16}.

\subsection{Hamiltonian operator}

We equip the metric graph $\mathcal G$ with the operator $\operatorname{H} : L^2(\mathcal G) \to L^2(\mathcal G)$, called the \emph{Hamiltonian operator}, acting as
\[
    \operatorname{H} \psi = (-\psi_e'')_{e \in \mathcal E}
\]
and defined on a subdomain of $H^2(\mathcal G)$ including local vertex conditions of the form
\begin{equation*} \label{eqVertexConditions}
    A_v \Psi(v) + B_v \Psi'(v) = 0
\end{equation*}
for all $v \in \mathcal V$, where $A_v$ and $B_v$ are $(d_v \times d_v)$-matrices. The domain of $\operatorname{H}$ is given by
\[
    D(\operatorname{H}) = \left\{ \psi \in H^2(\mathcal G): A_v \Psi(v) + B_v \Psi'(v) = 0 \text{ for all } v \in \mathcal V \right\}.
\]
A comprehensive introduction to the Hamiltonian operator is given in \cite[Chapter 1]{BeKu13}. In particular, a self-adjointness criterion for $\operatorname{H}$ is the following. 
\begin{proposition}\label{prpHSelfAdj}
    The following assertions are equivalent:
    \begin{enumerate}
        \item The operator $\operatorname{H}$ is self-adjoint.
        \item For every $v \in \mathcal V$, the matrices $A_v$ and $B_v$ satisfy the following conditions: the $(d_v \times 2d_v)$-matrix $(A_v, B_v)$ has maximal rank and the matrix $A_v B_v$ is symmetric.
        \item For every $v \in \mathcal V$, there are three orthogonal and mutually orthogonal projectors $P_{D,v}, P_{N,v}$ and $P_{R,v} = I_{d_v} - P_{D,v} - P_{N,v}$ acting on $\mathbb C^{d_v}$ and an invertible symmetric matrix $\Lambda_v$ acting on the subspace $P_{R,v} \mathbb C^{d_v}$ such that, for $\psi \in D(\operatorname{H})$,
        \[
            P_{D,v} \Psi(v) = 0,
            \quad P_{N,v} \Psi'(v) = 0,
            \quad P_{R,v} \Psi'(v) = \Lambda_v P_{R,v} \Psi(v).
        \]
    \end{enumerate}
\end{proposition}
In what follows, the operator $\operatorname{H}$ is always supposed to be self-adjoint. Thus, by the third point of Proposition \ref{prpHSelfAdj}, its quadratic form $Q_{\H}$ can be expressed as
\begin{align}\label{eqQuadFormDef}
    Q_{\H}(\psi) & = \langle \operatorname{H} \psi, \psi \rangle_{L^2(\mathcal G)}  = \| \psi' \|_{L^2(\G)}^2 + \sum_{v \in \mathcal V} \langle \Lambda_v P_{R,v} \Psi(v), P_{R,v} \Psi(v) \rangle_{\mathbb C^{d_v}}.
\end{align}
It is defined on the domain
\begin{equation}
    \label{eqDomainQ}
    D(Q_{\H}) = \left\{ \psi \in H^1(\mathcal G): P_{D, v} \Psi(v) = 0 \text{ for all } v \in \mathcal V \right\}.
\end{equation}
Observe that the quadratic form comprises two parts: the $L^2(\mathcal G)$-norm of the derivative of the function and a quadratic part coming from the vertex conditions. The following lemma holds.
\begin{lem} \label{lemStrongConv}
    Let $(\psi_n)_{n \in \mathbb N} \subset D(Q_{\H})$ be a sequence bounded in $H^1 (\mathcal G)$. Then, there exists $\psi \in D(Q_{\H})$ such that $\psi_n \rightharpoonup \psi$ as $n \to \infty$ weakly in $H^1(\mathcal G)$ up to a subsequence. Furthermore, for $v \in \mathcal V$, we have
    \begin{equation} \label{eqStrConv}
         \sum_{v \in \mathcal V} \langle \Lambda_v P_{R,v} \Psi_n (v), P_{R,v} \Psi_n (v) \rangle_{\mathbb C^{d_v}} \to \sum_{v \in \mathcal V} \langle \Lambda_v P_{R,v} \Psi (v), P_{R,v} \Psi (v) \rangle_{\mathbb C^{d_v}} \text{ as $n \to \infty$}
    \end{equation}
    up to a subsequence.
\end{lem}
\begin{proof}
    Let $\psi \in H^1(\mathcal G)$ be such that $\psi_n \rightharpoonup \psi$ as $n \to \infty$ weakly in $H^1(\mathcal G)$ up to a subsequence. Let $v \in \mathcal V$ and $e \in J_v$. We identify $e$ with $[v, A]$ or $[v,\infty)$. The compact embedding $H^1(v, L) \subset C\left( [v, L] \right)$ holds, where $L$ has been defined in \eqref{eqMinLenght}. Therefore, we have $\psi_{n,e} (v) \to \psi_e (v)$ as $n \to \infty$ up to a subsequence, so that $\Psi_n(v) \to \Psi(v)$ as $n \to \infty$. In particular, $P_{D,v} \Psi(v) = 0$, so $\psi \in D(Q_{\H})$. As
    \[
       \left| \langle \Lambda_v P_{R,v} \Psi(v), P_{R,v} \Psi(v) \rangle_{\mathbb C^{d_v}} \right| \lesssim \sum_{v \in \mathcal{V}} \sum_{e, e' \in J_v} |\psi_e(v)| |\psi_{e'}(v)|,
    \]
    we obtain \eqref{eqStrConv}, and this concludes the proof.
\end{proof}
We also show that the essential spectrum of $\operatorname{H}$ is given by $[0,\infty)$. This result can be proven by the classical Weyl's theorem \cite[Theorem $6.19$]{Te14} and has been shown under several different settings, as for example for star-graphs in \cite[Proposition $3.2$]{GrIg19}. We present a slightly alternative proof that covers all the relevant settings in our work. The idea of the proof is to compute first the essential spectrum of the Hamiltonian operator with Dirichlet-vertex conditions and then to deduce the essential spectrum of $\operatorname{H}$ through the computation of a Weyl sequence.
\begin{prop} \label{prpEssSpec}
    The essential spectrum of $\operatorname{H}$ is given by $\sigma_{ess} ( \operatorname{H} ) = [0, \infty)$.
\end{prop}
\begin{proof}
    \textit{Step 1}: We introduce the Hamiltonian operator $\operatorname{H}_D$ with \emph{Dirichlet-vertex conditions}, defined on the domain
    \[
        D( \operatorname{H}_D ) = \left\{ \psi \in H^2( \mathcal G ) : \psi_e (v) = 0 \text{ for all } v \in \mathcal V, e \in J_v \right\}.
    \]
    This operator is self-adjoint (see \cite[Chapter $1$]{BeKu13}) and, for $\psi \in D( \operatorname{H}_D )$, we have
    \[
        \langle \operatorname{H}_D \psi, \psi \rangle_{L^2 ( \mathcal G )} = \operatorname{Re} \left( \int_{\mathcal G} | \psi' (x) |^2 dx \right) \geq 0.
    \]
    Thus, $\sigma( \operatorname{H}_D ) \subset [0, \infty).$ To show the reverse inclusion, we consider $\lambda \in \rho( \operatorname{H}_D )$, where $\rho( \operatorname{H}_D )$ denotes the resolvent set of $\operatorname{H}_D$. For all $\psi \in L^2 ( \mathcal G )$, there exists a unique $\phi \in D( \operatorname{H}_D )$ such that, for all $e \in \mathcal E$,
    \[
        ( - \partial_{xx} - \lambda )^{-1} \psi_e (x) = \phi_e (x), \quad x \in I_e.
    \]
    Let $e \in \mathcal E$ such that $L_e = \infty$. We extend $\psi_e$ (respectively $\phi_e$) by central symmetry on $\mathbb R$, and we denote its extension by $\tilde \psi_e$ (respectively $\tilde \phi_e$), that is 
    \begin{equation}
        \tilde \psi_e(x) = \begin{cases}
            \psi_e(x) \quad \ \ \, x \geq 0, \\
            \psi_e(-x) \quad x < 0,
        \end{cases}
    \end{equation}
    and the same for $\tilde \phi_e$ so that $\tilde \psi_e \in L^2(\mathbb R)$ and $\tilde \phi_e \in H^2( \mathbb R )$. Then, $\tilde \phi_e$ is the only function in $H^2( \mathbb R )$ (up by multiplication by a constant) which satisfies
    \[
        (-\partial_{xx} - \lambda)^{-1} \tilde \psi_e (x) = \tilde \phi_e (x), \quad x \in \mathbb R,
    \]
    otherwise this would contradict the fact that $\phi$ is the unique preimage of $\psi$ under the operator $(\operatorname{H}_D - \lambda)^{-1}$. In particular, the unbounded operator $(-\partial_{xx} - \lambda)$ defined on the domain $\left\{ \psi \in H^2( \mathbb R ) : f(0) = 0 \right\} \subset L^2 ( \mathbb R )$ is invertible. This operator has no discrete spectrum and its essential spectrum is given by $[0, \infty)$. 
    Therefore, $\lambda \in \mathbb C \setminus [0, \infty)$ and finally
    \begin{equation} \label{eqEssSpecDir}
        \sigma_{ess} ( \operatorname{H}_D ) = [0, \infty).
    \end{equation}
    \textit{Step 2}: Let $\lambda \in \sigma_{ess} ( \operatorname{H} )$. From \cite[Chapter $9$]{EdEv18}, there exists a Weyl's sequence $(\phi_n)_{n \in \mathbb N} \subset D( \operatorname{H} )$, i.e. the sequence satisfies the following conditions:
    \begin{align}
        & \| \phi_n \|_{L^2 ( \mathcal G )} = 1 \text{ for all } n \in \mathbb N, \label{eqWeyl1} \\
        & \phi_n \rightharpoonup 0 \text{ weakly in } L^2( \mathcal G) \text{ as } n \to \infty, \label{eqWeyl2} \\
        & ( \operatorname{H} - \lambda ) \phi_n \to 0 \text{ in } L^2( \mathcal G) \text{ as } n \to \infty. \label{eqWeyl3}
    \end{align}
    As \eqref{eqWeyl1} holds, observe that
    \[
        \| \operatorname{H} \phi_n \|_{L^2 ( \mathcal G )}^2 = \| \lambda \phi_n \|_{L^2 ( \mathcal G )}^2 + o(1) = |\lambda|^2 + o(1),
    \]
    so $(\phi_n)_{n \in \mathbb N}$ is bounded in $H^2 ( \mathcal G )$. Consider a function $\eta \in C^\infty ( \mathcal G )$ such that, for $e \in \mathcal E$,
    \begin{align*}
        \eta_e (x) & = 1 \text{ for } x \in (0, L_e) & \text{ if } L_e < \infty, \\
        \eta_e (x) & =
        \begin{cases}
            1 \text{ for } x \in (0, L) \\
            0 \text{ for } x \in (2L, \infty)
        \end{cases}
        & \text{ if } L_e = \infty,
    \end{align*}
    where $L$ is defined in \eqref{eqMinLenght}.
    Observe that $\eta$ is compactly supported in $\mathcal G$. The sequence $(\eta \phi_n)_{n \in \mathbb N}$ is bounded in $H^2( \mathcal G )$, thus is bounded in $L^\infty ( \mathcal G )$ by Sobolev embeddings. We have
    \[
        \| \eta \phi_n \|_{L^2 ( \mathcal G )}^2 \leq \| \eta \phi_n \|_{L^\infty (\mathcal G)} \langle \eta, \phi_n \rangle_{L^2 ( \mathcal G )} \to 0 \text{ as } n \to \infty.
    \]
    by \eqref{eqWeyl2}, which implies that
    \begin{equation} \label{eqWeyl21}
        \| (1 - \eta) \phi_n \|_{L^2 ( \mathcal G )}^2 \to 1 \text{ as } n \to \infty.
    \end{equation}
    Furthermore, for any $\psi \in L^2( \mathcal G )$, we have
    \[
        | \langle (1 - \eta) \phi_n, \psi \rangle_{L^2 ( \mathcal G )} | \leq \| \psi \|_{L^2 ( \mathcal G )}^2 \| \eta \phi_n \|_{L^2 ( \mathcal G )}^2 + | \langle \phi_n, \psi \rangle_{L^2 ( \mathcal G )} | \to 0 \text{ as } n \to \infty,
    \]
    so
    \begin{equation} \label{eqWeyl22}
        (1 - \eta) \phi_n \rightharpoonup 0 \text{ weakly in } L^2 ( \mathcal G ) \text{ as } n \to \infty.
    \end{equation}
    On each edge $e \in \mathcal E$, we have
    \begin{equation}\label{eqCommut}
    \begin{aligned}
         ( \partial_{xx} + \lambda ) ( \eta_e \phi_{n,e} ) & = \partial_{xx} \eta_e \phi_{n,e} + 2 \partial_x \eta_e \partial_x \phi_{n,e} + \eta_e \partial_{xx} \phi_{n,e} + \lambda \eta_e \phi_{n,e} \\
         &= \eta_e (\partial_{xx} + \lambda) \phi_{n,e} + [\partial_{xx} + \lambda, \eta_e]  \phi_{n,e}.
    \end{aligned}
    \end{equation}
    Since $\phi_{n,e}$ is uniformly bounded in $H^2(I_e)$ and $\phi_{n,e} \rightharpoonup 0$ as $n \to \infty$ in $L^2(I_e)$, then, up to extracting a subsequence, it follows that $\phi_{n,e} \rightharpoonup 0$ in $H^2(I_e)$. Since $ \partial_{xx} \eta_e, \partial_{x} \eta_e $ are smooth and supported in bounded domains, from \eqref{eqWeyl2}, \eqref{eqWeyl3} and \eqref{eqCommut}, we obtain that
    \begin{equation*}
        \begin{aligned}
             \| ( \partial_{xx} + \lambda )  ( \eta_e \phi_{n,e} ) \|_{L^2 ( \mathcal G )} \leq  \| \eta_e ( \partial_{xx} + \lambda ) \phi_{n,e} ) \|_{L^2 ( \mathcal G )} +  \|  [\partial_{xx} + \lambda, \eta_e]  \phi_{n,e} \|_{L^2 ( \mathcal G )}  \to 0 
        \end{aligned}
    \end{equation*}
   as $n \to \infty$. This implies that 
   \begin{equation*}
       \| (\H - \lambda)(\eta \phi_n) \|_{L^2(\G)}^2 \to 0 \ \mbox{ as } \ n\to \infty.
   \end{equation*}
   Using \eqref{eqWeyl3}, we have
    \begin{equation} \label{eqWeyl23}
        \begin{aligned}
            \| ( \operatorname{H} - \lambda ) ( 1 - \eta ) \phi_n \|_{L^2 ( \mathcal G )}^2
            & \lesssim \| ( \operatorname{H} - \lambda ) \phi_n \|_{L^2 ( \mathcal G )}^2 + \| ( \operatorname{H} - \lambda ) ( \eta \phi_n ) \|_{L^2 ( \mathcal G )}^2 \to 0 \text{ as } n \to \infty.
        \end{aligned}
    \end{equation}
    Let $( \varphi_n )_{n \in \mathbb N} \subset D( \operatorname{H} )$ be given by
    \[
        \varphi_n \coloneqq \frac{(1 - \eta) \varphi_n}{\| (1 - \eta)\varphi_n \|_{L^2 ( \mathcal G )}}, \quad n \in \mathbb N.
    \]
    By \eqref{eqWeyl21}, \eqref{eqWeyl22} and \eqref{eqWeyl23}, $( \varphi_n )_{n \in \mathbb N}$ is a Weyl's sequence for $\operatorname{H}$. For $v \in \mathcal V$, we have $\varphi_{n,e} = 0$ on the neighborhood of $v$ for each $e \in J_v$. Thus $(\phi_n) \subset D(\H_D)$ and
    \begin{align*}
        ( \operatorname{H}_D - \lambda ) \varphi_n & = ( \operatorname{H} - \lambda ) \varphi_n  \to 0 \text{ in } L^2 ( \mathcal G ) \text{ as } n \to \infty.
    \end{align*}
    Hence, $( \varphi_n )_{n \in \mathbb N}$ is also a Weyl sequence for $\operatorname{H}_D$. In particular, $\lambda \in \sigma_{ess} ( \operatorname{H}_D ) \subset [0,\infty)$ by \eqref{eqEssSpecDir}.
    We may show the reverse inclusion by switching the role $\operatorname{H}$ and $\operatorname{H}_D$. Therefore
    \[
        \sigma_{ess} ( \operatorname{H} ) = \sigma_{ess} ( \operatorname{H}_D ) = [0, \infty)
    \]
    and this concludes the proof.
\end{proof}
We also remark that the operator \( \H \) can have at most a finite number of negative eigenvalues, as established in \cite[Theorem 3.7]{KoSc06}. Each of these eigenvalues belongs to the discrete spectrum \( \sigma_{\text{dis}}(\H) \), since \( \H \) can be viewed as the Laplacian perturbed by a finite-rank operator determined by the vertex conditions.

\subsection{Well-posedness}

The Cauchy problem \eqref{eqTimeDepNLS} is globally well-posed on the energy space, and conserves mass and energy, as stated in the following proposition.
\begin{proposition} \label{propLocalWP}
    For $u_0 \in D(Q_{\H})$, there exists $T \in (0, \infty]$ such that equation \eqref{eqTimeDepNLS} has an unique global solution $u \in C( [0, \infty), D(Q_{\H}) ) \cap C^1( [0, T), D(Q_{\H})^\star )$. 
    Furthermore, the solution $u$ satisfies the conservation of mass and energy, i.e.
    \[
        M (u(t, \cdot)) = M (u_0), \quad E (u(t, \cdot)) = E (u_0)
    \]
    for all $t \in [0, \infty)$.
\end{proposition}

The proof of this result follows standard arguments; see \cite[Chapter 4]{Ca03} for reference. Using the conservation of mass and energy, along with the defocusing nature of the nonlinearity, we deduce that the unique solution to the Cauchy problem exists globally.

\section{Existence of ground states}\label{secCriterion}

\subsection{Action ground states}

To enhance clarity, we structure the proof of Theorem \ref{thmCriterion} and Theorem \ref{thmCritCase} through a series of intermediate lemmas. We begin by proving that if  $\omega \geq 0$, then  $s_\omega > -\infty$. Moreover, on non-compact graphs, this condition is in fact equivalent:  $s_\omega > -\infty$  necessarily implies  $\omega \geq 0$. Next, we establish that $\omega < -l_{\operatorname{H}}$ if and only if  $s_\omega < 0$, leading to the conclusion that on non-compact graphs, the condition  $s_\omega \in (-\infty, 0)$  holds if and only if  $l_{\operatorname{H}} < 0$  and  $\omega \in [0, -l_{\operatorname{H}})$. We then demonstrate that, for $\omega \neq 0$ and for $\omega = 0, 1< p < 5$, an action ground state exists if and only if $s_\omega \in (-\infty, 0)$, which proves Theorem \ref{thmCriterion}. For $p \geq 5$, we show that if a solution exists, it has to vanish on the infinite edges, which proves Theorem \ref{thmCritCase}.

Finally, applying the classical Cazenave-Lions argument \cite{CaLi82}, we establish the orbital stability of the set of ground states.  We will also show some properties of $s_\omega$. 

We start with the lower bound on $s_\omega$.
\begin{lem}\label{lem2}
    If $\omega \geq 0$, then $s_\omega > -\infty$.
\end{lem}
\begin{proof}
    Let $v \in \mathcal{V}$ and $e \in J_v$. We identify it with $e \subseteq [v, w_e)$  with $v \in \R$ and $w_e \in (v, \infty]$.  By the definition of  $L$  in \eqref{eqMinLenght}, it follows that  $[v, v+L] \subset e$.
    Then for any $\psi \in D(Q_{\operatorname{H}})$ and any $x \in [v, v+ L]$, we have
    \begin{equation*}
        | \psi_e (v) |^2 = |\psi_e(x)|^2 + \int_v^{x} \partial_y |\psi_e(y)|^2dy \leq |\psi_e(x)|^2 + 2 \int_v^{x} |\psi_e'(y) \psi_e(y)|dy.
    \end{equation*}
    It follows from Young's inequality that for any $\eps >0$ we get
    \begin{equation*}
        | \psi_e (v) |^2 = \frac{1}{L} \int_v^{v+L} | \psi_e (v) |^2 dx \leq \frac{C}{\eps} \| \psi_e \|_{L^2(v,v+L)}^2 + \eps \| \psi_e' \|_{L^2(v,v+L)}^2.
    \end{equation*}
    where $C>0$ depends only on $v$ and $L$. Then the triangular inequality implies that 
    \begin{align*}
        \langle \Lambda_v P_{R,v} \Psi(v), P_{R,v} \Psi(v) \rangle_{\mathbb C^{d_v}} & \lesssim \sum_{v\in\mathcal{V}} \sum_{e,e' \in J_v}|\psi_e(v)||\psi_{e'}(v)| \\
        &\lesssim \sum_{v \in \mathcal{V}} \sum_{e \in J_v}| \psi_e(v) |^2 \leq \sum_{v\in \mathcal{V}} \sum_{e \in J_v} \frac{C}{\eps} \| \psi_e \|_{L^2(v,v+L)}^2 + \eps \| \psi_e' \|_{L^2(v,v+L)}^2.
    \end{align*}
    Now we use Hölder's and Young's inequalities to obtain that there exists $R = R(L) > 0$ such that
    \[
        \frac{C}{\eps} \| \psi_e \|_{L^2(v,v+L)}^2 \leq \frac{1}{2(p+1)} \| \psi_e \|_{L^{p+1}(v,v+L)}^{p+1} + \frac{R}{\eps}.
    \]
    Since $|\mathcal{V}| < \infty$ and $|\mathcal{E}| <\infty$, there exists $K = K(|\mathcal{V}|, |\mathcal{E}|) > 0$ such that we obtain the lower bound on the action 
    \begin{equation}\label{eqBound2}
        S_\omega(\psi) \geq \left( \frac{1}{2} - \eps \right) \| \psi' \|_{L^2(\G)}^2 + \frac{\omega}{2} \| \psi \|_{L^2(\G)}^2 +  \frac{1}{2(p+1)} \| \psi \|_{L^{p+1}(\G)}^{p+1}  - \frac{KR}{\eps}
    \end{equation}
    which yields the result for $\eps \in (0, 1/2).$ 
\end{proof}

We show the opposite implication for non-compact graphs.
\begin{lem}\label{lemOmNeg}
    Suppose that $\G$ is non-compact. If $\omega < 0$, then $s_{\omega} = -\infty$.
\end{lem}

\begin{proof}
    Let $e$ be one of the infinite edges, identified with $e \cong [0,\infty)$. Let $\psi \in D(Q_{\operatorname{H}})$ be such that $\psi \not \equiv 0$ on $e$, $\psi(0) = 0$ and $\psi \equiv 0$ any other edge of the graph. Then for any $\lambda >0$, 
    \begin{equation*}
        2 S_{\omega}\left(\lambda^{1/2} \psi\left(\lambda^{\frac{p+1}{2}} x \right)\right) = \lambda^\frac{p+3}{2} \| \partial_x \psi \|_{L^2(\G)}^2 + \omega \lambda^\frac{1-p}{2} \| \psi \|_{L^2(\G)}^2 + \frac{2}{p+1}   \| \psi \|_{L^{p+1}(\G)}^{p+1} \to -\infty
    \end{equation*}
    as $\lambda \to 0$, since $p >1$.
\end{proof}

We proceed by obtaining an upper bound for $s_\omega$.
\begin{lem} \label{lem4}
    We have $\omega < - l_{\H}$ if and only if $s_\omega <0$.
\end{lem}
\begin{proof}
    Suppose  $\omega < - l_{\H}$. We consider first the case $l_{\H} \in \sigma_{dis}(\H)$.   Let $\psi$ be such that $\| \psi \|_{L^2(\G)}^2 = 1$ and  
    \begin{equation*}
        \H \psi = l_{\H} \psi.
    \end{equation*}
    Then for any $\mu \in \R$
    \begin{multline*}
        S_\omega(\mu \psi) = \mu^2 \left( Q_{\H}(\psi) - l_{\H} \| \psi \|_{L^2(\G)}^2 \right) + \mu^2 (\omega + l_{\H}) \| \psi \|_{L^2(\G)}^2 + \frac{\mu^{p+1}}{p+1}\| \psi \|_{L^{p+1}(\G)}^{p+1} \\
        = \mu^2 (\omega + l_{\H}) \| \psi \|_{L^2(\G)}^2 + \frac{\mu^{p+1}}{p+1}\| \psi \|_{L^{p+1}(\G)}^{p+1} \to 0^{-}
    \end{multline*}
    as $\mu \to 0$ since $\omega + l_{\H} <0$ and $p > 1$. Thus $s_\omega < 0$.
    
    On the other hand, if $l_{\H} \in \sigma_{ess}(\H)$, then let $(\psi_n)_{n \in \mathbb N}$ be a minimizing sequence for $l_{\H}$. Since $Q(\psi_n) \to l_{\H}$ as $n \to \infty$, one can repeat the same proof by take $\psi_n$ with $n$ big enough.
    
    For the opposite implication, suppose $s_{\omega} <0$. Then there exists $\psi$ such that $S_\omega(\psi) <0$. Then 
    \begin{equation*}
        Q_{\H}(\psi) - l_{\H}\| \psi \|_{L^2(\G)}^2 + (\omega + l_{\H}) \| \psi \|_{L^2(\G)}^2 < 0. 
    \end{equation*}
    By definition of $l_{\H}$, $Q_{\H}(\psi) - l_{\H}\| \psi \|_{L^2(\G)}^2 \geq 0$. Thus $\omega < - l_{\H}$.
\end{proof}

Finally, we are ready to prove the criterion stated in Theorem \ref{thmCriterion}. We will separate the cases $\omega > 0$ and $\omega = 0$ as the second case is more complicated. We start with the first in the following Lemma. 
\begin{lem}\label{lemExisGSAc}  
    Let $\G$ be a non-compact graph and $\omega \neq 0$. If $s_\omega \in (-\infty,0)$, then a minimizer of \eqref{eqMinAction} exists.
\end{lem}

\begin{proof}
    Suppose that $s_\omega \in (-\infty,0)$ and $\omega \neq 0$. This implies $\omega \in (0,-l_{\operatorname{H}})$ by Lemmas \ref{lemOmNeg} and \ref{lem4}. Let $(\psi_n)_{n \in \mathbb N} \subset D(Q_{\operatorname{H}})$ be a minimizing sequence for $s_{\omega} < 0$. Then from \eqref{eqBound2} we get 
    \begin{equation*}
        s_\omega + C \geq \frac{1}{4} \| \psi_n' \|_{L^2(\G)}^2 + \frac{\omega}{2} \| \psi_n \|_{L^2(\G)}^2,
    \end{equation*}
    that is there exists $C > 0$ such that $\| \psi_n \|_{H^1(\G)} \leq C$ for all $n \in \mathbb N$.  By Lemma \ref{lemStrongConv}, there exists $\psi \in D(Q_{\operatorname{H}})$ such that, up to a subsequence, $\psi_n \rightharpoonup \psi$ in $H^1(\G)$ as $n \to \infty$. In particular \eqref{eqStrConv} holds for any $v \in \mathcal V$, which implies that
    \[
        S_{\omega}(\psi) \leq s_{\omega} = \liminf_{n \to \infty} S_{\omega}(\psi_n) < 0
    \]
    by weak lower semi-continuity. In particular $\psi \not \equiv 0$ and, by the definition of $s_{\omega}$, $S_{\omega}(\psi) \geq s_{\omega}$. Thus $S_{\omega}(\psi) = s_{\omega}$ and $\psi_n \to \psi$ strongly in $H^1(\G)$ as $n \to \infty$ by lower semi-continuity.
\end{proof}
We now turn to the case $\omega = 0$. The key difference here is that \eqref{eqBound2} no longer directly provides a bound on the $L^2$ norm, making it nontrivial to show that the weak limit of the minimizing sequence lies in $H^1(\G)$. To overcome this problem, we first work in the intermediate weaker space $\dot{H}^1(\G) \cap L^{p+1}(\G)$ and establish the existence of a genuine minimizer there. We then show that this minimizer belongs to $L^2(\G)$ for $p < 5$, while for $p \geq 5$ it does only if the solution is supported on the compact core.
\begin{lem}\label{lemNew}
     Let $p \in (1,5)$. Then $s_0$ admits a minimizer
\end{lem}
\begin{proof}
    First, we slightly change the minimizing problem. Let us define 
    \begin{equation}
        \label{eqMinProb4}
        \tilde s = \inf \{S_0(u) : u \in \dot{H}^1(\G) \cap L^{p+1}(\G) \}
    \end{equation}
    Let $(\psi_n)_{n \in \mathbb N}$ be a minimizing sequence for $\tilde s$. By \eqref{eqBound2}, we see that this sequence is bounded in the $\dot{H}^1(\G)$ norm and the $L^{p+1}(\G)$, and the space $\dot{H}^1(\G) \cap L^{p+1}(\G)$ is reflexive. Thus, up to subsequences, there exists a weak limit $\psi \in \dot{H}^1(\G) \cap L^{p+1}(\G)$ such that $\psi_n \rightharpoonup \psi$ as $n \to \infty$ in $\dot{H}^1(\G) \cap L^{p+1}(\G)$. Notice that $(\psi_n)_{n \in \mathbb N} \subset H^1_{loc}(\G)$. By the theorem of Ascoli-Arzela, this implies that, up to subsequences, we get 
    \begin{equation*}
        \sum_{v \in \mathcal V} \langle \Lambda_v P_{R,v} \Psi_n(v), P_{R,v} \Psi_n(v) \rangle_{\mathbb C^{d_v}} \to \sum_{v \in \mathcal V} \langle \Lambda_v P_{R,v} \Psi(v), P_{R,v} \Psi(v) \rangle_{\mathbb C^{d_v}}
    \end{equation*}
    as $n \to \infty$, that is the contribution of the vertex conditions in the quadratic form defined in \eqref{eqQuadFormDef} converges point-wise. Thus, by the weak lower-semicontinuity of the $\dot{H}^1(\G)$ and $L^{p+1}(\G)$ norms, we get 
    \begin{equation*}
        S_0(\psi) \leq \liminf_{n \to \infty} S_0(\psi_n) = \tilde s
    \end{equation*}
    which yields $S_0(\psi) = \tilde s$.
    
    Now we show that $\psi \in \dot{H}^1(\G)$ and $L^{p+1}(\G)$ found above is in $H^1(\G)$ if and only if $ p \in (1,5)$. Indeed, notice that $\psi$ is still a solution to \eqref{eqStationaryNLS} for $\omega =0$, by the classical Euler-Lagrange theory. Moreover, being in $L^{p+1}(\mathcal G)$, it is still decaying at infinity on the infinite edges. The only (real and positive) decaying solutions are of the form
    \begin{equation}\label{eqSol2}
            \psi_e (x) = \left( \frac{1-p}{\sqrt{2(p+1)}} x + b \right)^{-\frac{2}{p-1}}
        \end{equation}
    for some  $b = b(p, \operatorname{H}) > 0$  depending on the vertex condition, where the infinite edge $e$ is identified with $[0,\infty)$. This can be proven similarly to Theorem $2$ in  \cite{KaOh09}, where it is also shown that these solutions are unique up to a phase shift in the complex plane. Now, by a direct computation, one can check that $\psi_e$ belongs to $L^2(\R^+)$ for $p \in (1,5)$. As a consequence, for $p \in (1,5)$, $\psi \in H^1(\G)$. As $H^1(\G) \subset \dot{H}^1(\G) \cap L^{p+1}(\G) $, we have a priori that $\tilde {s} \leq s_0$, we thus obtain $s_0 = \tilde s$ for $p \in (1,5)$ which concludes the proof. 
\end{proof}
Finally, we observe that the obstruction for $p \geq 5$ arises from the fact that the unique decaying solution at infinity, given by \eqref{eqSol2}, does not belong to $L^2(\mathbb{R}^+)$. Consequently, the only possibility for a nontrivial $H^1(\G)$ solution to \eqref{eqStationaryNLS} in this regime is to be compactly supported, i.e., entirely contained within the compact core and localized on edges of finite length.

\begin{proof}[Proof of Theorem \ref{thmCritCase}]
    The proof follows closely that of Lemma \ref{lemNew} above. Indeed, by direct computation, one can see that the unique nontrivial solution allowed on infinite edges in \eqref{eqSol2} is not in $L^2(\R^+)$. Thus, either the solution is not in  $H^1(\G)$, or it has to be equal to zero on infinite edges. 
\end{proof}

We show the opposite implication of the criterion stated in Theorem \ref{thmCriterion}. 
\begin{lem}\label{lemOppos}
    Let $\G$ be a non-compact graph. If a minimizer of \eqref{eqMinAction} exists, then $s_\omega \in (-\infty, 0)$.
\end{lem}

\begin{proof}
    Let $\psi$ be a ground state. By contradiction, suppose first that $s_{\omega} > 0$. Observe that for any $\mu >0$,
    \begin{equation*}
        S_{\omega}(\mu \psi) = \frac{\mu^2}{2} Q_{\H}(\psi) + \frac{\omega \mu^2}{2} \| \psi\|_{L^2(\G)}^2  + \frac{\mu^{p+1}}{p+1} \| \psi \|_{L^{p+1}(\G)}^{p+1}  \to 0
    \end{equation*}
    as $\mu \to 0$. This contradicts the definition of $ s_{\omega}$. Therefore $s_\omega < 0$.
    
    Now suppose instead $s_{\omega} = 0$ and $\psi \not \equiv 0$. Then observe that
    \begin{equation*}
        S_{\omega}(\mu \psi) = \frac{\mu^2}{2} S_{\omega}(\psi) - \frac{\mu^2 - \mu^{p+1}}{p+1} \| \psi \|_{L^{p+1}(\G)}^{p+1} = - \frac{\mu^2 - \mu^{p+1}}{p+1} \| \psi \|_{L^{p+1}(\G)}^{p+1}
    \end{equation*}
    Then either $\psi \equiv 0$ which is excluded from the definition of $s_{\omega}$ or there exists $\mu^* >0$ such that for any $0 < \mu < \mu^*$, $S_{\omega}(\mu \psi) < 0$. The second case leads to a contradiction with the definition of $s_{\omega}$.
\end{proof}

Finally, we prove our main result.
\begin{proof}[Proof of Theorem \ref{thmCriterion}]
    The proof is a combination of Lemmas  \ref{lem2}, \ref{lemOmNeg}, \ref{lem4}, \ref{lemExisGSAc}, \ref{lemNew} and \ref{lemOppos}.
\end{proof} 

As the action ground states are unconstrained minimizers of the action, which is preserved along the flow of equation \eqref{eqTimeDepNLS}, we can deduce the stability of the set of action ground states.
\begin{lem}\label{lemStabGSAc}
    The set of action ground states $\mathcal A_\omega$ is stable.
\end{lem}
\begin{proof}
    It is a standard result, in the spirit of the orbital stability obtained in \cite{CaLi82}. We report the short proof. Suppose, by contradiction, that the set of grounds states is not stable. Thus there exists $\eps >0$ and $(u_0^{(n)})_{n \in \mathbb N} \subset D(Q_{\operatorname{H}})$ a set of initial condition such that
    \begin{equation*}
        \inf_{\phi_\omega \in \mathcal{A}_\omega} \| u_0^{(n)} - \phi_\omega \|_{H^1(\mathcal G)} \to 0 \text{ as } n \to \infty
    \end{equation*}
    while there exists $(t_n)_{n \in \mathbb N} \subset \R$ such that 
    \begin{equation}\label{eqAbs1}
        \inf_{\phi_\omega \in \mathcal{A}_\omega} \| u_{n}(t_n) - \phi_\omega \|_{H^1(\mathcal G)} \geq \eps \text{ for all } n \in \mathbb N,
    \end{equation}
    where $(u_n)_{n \in \mathbb N}$ are solutions of \eqref{eqTimeDepNLS} stemming from $u_0^{(n)}$. Due to the conservation of the action, we get $S_\omega(u_n(t_n)) = S_\omega(u_0^{(n)}) \to s_\omega$ as $n \to \infty$. Thus $(u_n(t_n))_{n \in \mathbb N}$ is a minimizing sequence for $s_\omega$. Then with the same argument as in the proof of Lemma \ref{lemExisGSAc}, one gets that $u_n(t_n)$ converges strongly to one action ground state up to a subsequence, which violates \eqref{eqAbs1}.
\end{proof}

Next, we study the function $\omega \mapsto s_\omega$.
\begin{lem}\label{prpSOmega}
    Assume $l_{\H} <0$ and $\omega \in [0, -l_{\operatorname{H}})$ for $p \in (1,5)$, or $\omega \in (0,-l_{\H})$ for $p \geq 5$. The function $\omega \mapsto s_\omega$ is continuous and strictly increasing. Moreover, for any $\phi_\omega \in \mathcal A_\omega$, we have
        \begin{equation}\label{eqSOmegaPot}
            s_\omega = - \frac{p-1}{2(p+1)} \| \phi_\omega \|_{L^{p+1}(\G)}^{p+1} = \frac{p-1}{2(p+1)} \left( Q_{\operatorname{H}}(\phi_\omega) + \omega \| \phi_\omega\|_{L^2(\G)}^2 \right).
        \end{equation}
        In particular for $\omega_1 < \omega_2$, $\phi_{\omega_1} \in \mathcal A_{\omega_1}$, $\phi_{\omega_2} \in \mathcal A_{\omega_2}$,  we get $\| \phi_{\omega_1} \|_{L^{p+1}(\G)} > \| \phi_{\omega_2} \|_{L^{p+1}(\G)} $.
\end{lem}

\begin{proof}
    Let $\omega_1,\omega_2 \in (0,-l_{\operatorname{H}})$ such that $\omega_1 < \omega_2$, and $\phi_{\omega_1} \in \mathcal{A}_{\omega_1}$, $\phi_{\omega_2} \in \mathcal{A}_{\omega_2}$. Then observe that
    \begin{equation}\label{eqSComparison1}
        S_{\omega_1}(\phi_{\omega_1}) <  S_{\omega_1}(\phi_{\omega_2}) < S_{\omega_2}(\phi_{\omega_2}) < S_{\omega_2}(\phi_{\omega_1}).
    \end{equation}
    The equalities are excluded since one of the profiles would otherwise satisfy the stationary equation \eqref{eqStationaryNLS} for both $\omega_1$ and $\omega_2$. This implies that $s_{\omega_2} > s_{\omega_1}$.

    Moreover, we have 
    \begin{equation}\label{eqSOmegaCont}
        |  S_{\omega_2}(\phi_{\omega_2}) -  S_{\omega_1}(\phi_{\omega_1}) | < |S_{\omega_2}(\phi_{\omega_1}) - S_{\omega_1}(\phi_{\omega_1})| = | \omega_2 - \omega_1| \| \phi_{\omega_1}\|_{L^2(\G)}^2 \to 0
    \end{equation}
    as $\omega_2 \to \omega_1$.

    We observe that, for any $\omega\in (0, -l_{\H})$, we have 
    \begin{equation} \label{eqEquiv1}
        \begin{aligned}
            s_{\omega}
            = \inf_{\psi \in D(Q_{\operatorname{H}})} S_\omega (\psi)
            & = \inf_{\psi \in D(Q_{\operatorname{H}}), I_\omega (\psi) = 0} S_\omega (\psi) \\
            & = \inf_{\psi \in D(Q_{\operatorname{H}}), I_\omega (\psi) = 0} - \frac{p-1}{2(p+1)} \| \psi \|_{L^{p+1}(\G)}^{p+1} \\
            & = \inf_{\psi \in D(Q_{\operatorname{H}}), I_\omega (\psi) = 0} \frac{p-1}{2(p+1)} \left( Q_{\operatorname{H}}(\psi) + \omega \| \psi \|_{L^2(\G)}^2 \right),
        \end{aligned}
    \end{equation}
    where 
    \begin{equation*}
        I_\omega(u) = Q_{\operatorname{H}}(\psi) + \| \psi \|_{L^{p+1}(\G)}^{p+1} + \omega \| \psi \|_{L^2(\G)}^2.
    \end{equation*}
    Indeed, the second equality in \eqref{eqEquiv1} follows from the fact that any $\psi \in \mathcal{A}_{\omega}$ satisfies $I_\omega(\psi) = 0$. The other inequalities follow from combining $S_\omega(\psi)$ with $I_\omega(\psi)$. 
\end{proof}

\begin{proof}[Proof of Corollary \ref{corStability}]
    The proof is a combination of Lemmas \ref{lemStabGSAc} and \ref{prpSOmega}.
\end{proof} 

\subsection{Energy ground states}

We turn our attention to the set of energy ground states. We show first that action ground states are also energy ground states. 

\begin{lem}\label{lemAcToEn}
   Let $\phi_\omega \in \mathcal{A}_\omega$, and $m = M(\phi_\omega)$. Then $\phi_\omega \in \mathcal{B}_m$,
\end{lem}

\begin{proof}
    Let $\phi_\omega$ be an action ground state for some $\omega \in \R$. Suppose by contradiction that there exist $\psi \in D(Q_{\operatorname{H}})$ such that $\| \psi \|_{L^2}= \| \phi_\omega\|_{L^2}$ and $E(\psi) < E(\phi_\omega)$. As an immediate consequence, we get 
    \begin{equation*}
        S_\omega(\psi) < S_\omega(\phi_\omega)
    \end{equation*}
    which contradicts the definition of $\phi_\omega$ as an action minimizer.
\end{proof}

\begin{lem} \label{prpProperties1}
    Let $\phi_1,\phi_2 \in \mathcal A_\omega$. If $ \| \phi_1 \|_{L^2(\G)}^2 = m_1 < m_2 = \| \phi_2 \|_{L^2(\G)}^2$, then $e_{m_1} > e_{m_2}$. Moreover, suppose that for some $m >0$, there exist $\psi_1 \neq \psi_2$, $\psi_1, \psi_2 \in \mathcal B_m$ and the associated Lagrange multipliers $\omega_1 \neq \omega_2$. Then either $\psi_1 \notin \mathcal A_{\omega_1}$ or $\psi_2 \notin \mathcal A_{\omega_2}$.
\end{lem}

\begin{proof}
    Suppose that $\phi_1,\phi_2 \in \mathcal A_\omega$ are such that $ \| \phi_1 \|_{L^2(\G)}^2 = m_1 < m_2 = \| \phi_2 \|_{L^2(\G)}^2$.  Observe that from \eqref{eqSOmegaPot} we directly obtain that $Q_{\operatorname{H}}(\phi_1) > Q_{\operatorname{H}}(\phi_2)$. Since, by Lemma \ref{lemAcToEn}, $\phi_1 \in \mathcal{B}_{m_1},\phi_2 \in \mathcal{B}_{m_2}$, we get 
    \begin{equation*}
        2e_{m_1} = Q_{\operatorname{H}}(\phi_1) + \frac{2}{p+1}\| \phi_1 \|_{L^{p+1}(\mathcal G)}^{p+1} >  Q_{\operatorname{H}}(\phi_2) + \frac{2}{p+1}\| \phi_2 \|_{L^{p+1}(\mathcal G)}^{p+1}  = 2e_{m_2}.
    \end{equation*}
    
    Now, suppose that for some $m >0$, there exists $\psi_1 \neq \psi_2$, $\psi_1, \psi_2 \in \mathcal B_m$ and the associated Lagrange multipliers $\omega_1 \neq \omega_2$. Without loss of generality, suppose $\omega_1 < \omega_2$. We get 
    \begin{equation*}
        2(S_{\omega_2}(\psi_2) - S_{\omega_1}(\psi_1)) = (\omega_2 - \omega_1) m.
    \end{equation*}
    But if both $\psi_1 \in \mathcal A_{\omega_1}$ and $\psi_2 \in \mathcal A_{\omega_2}$ then by \eqref{eqSComparison1}, we obtain 
    \begin{equation*}
        2(s_{\omega_2} - s_{\omega_1}) < 2 \left( S_{\omega_2}(\phi_{\omega_1}) - s_{\omega_1} \right)  = (\omega_2 - \omega_1) \| \phi_{\omega_1}\|_{L^2(\G)}^2 = (\omega_2 - \omega_1) m
    \end{equation*}
    which is a contradiction. 
\end{proof}

\begin{proof}[Proof of Theorem \ref{thmExisEnGs}]
The proof is given by combining Lemmas \ref{lemAcToEn} and \ref{prpProperties1}.
\end{proof}

\begin{remark} \label{rkMassCont}
We highlight the following observations:
\begin{enumerate}
\item If there exists a unique non-negative profile satisfying \eqref{eqStationaryNLS}, then the functions
\[
    \omega \in [0, -l_{\operatorname{H}}) \mapsto M(\phi_\omega) \text{ if } 1 < p < 5, \quad \omega \in (0, -l_{\operatorname{H}}) \mapsto M(\phi_\omega) \text{ if } p \geq 5,
\]
where  $\phi_\omega \in \mathcal{A}_\omega$, are injective and continuous, and thus admits a continuous inverse. Indeed, for any $\omega$ and any sequence $(\omega_n)_{n \in \mathbb N} \subset (0,-l_{\operatorname{H}})$  such that $\omega_n \to \omega$ as $n \to 0$, the corresponding real-valued action minimizers $\phi_{\omega_n}$ form a minimizing sequence for  $s_\omega$ by \eqref{eqSComparison1}. By compactness, up to a subsequence, they converge strongly in  $H^1(\mathcal{G})$  to a real-valued ground state  $\phi_{\omega} \in \mathcal{A}_\omega$. Uniqueness then ensures that  $\|\phi_{\omega_n}\|_{L^2(\G)} \to \|\phi_{\omega}\|_{L^2(\G)}$, proving continuity and invertibility.
\item In this case, the inverse of Theorem \ref{thmExisEnGs} holds, meaning that every energy ground state is also an action ground state.  
\item Conversely, if the function $m \mapsto \omega_m$, which associates a given mass to its corresponding Lagrange multiplier, is not injective, then by Lemma \ref{prpProperties1}, there exist energy minimizers that are not action minimizers. This contrasts, in general, with the focusing case. Moreover, non-injectivity of $m \mapsto \omega_m$ implies also that the existence of multiple non-negative solutions to \eqref{eqStationaryNLS}, even when accounting for symmetries of the graph.  
\end{enumerate}
\end{remark}

\section{Star graph with Delta and Delta Prime-vertex condition} \label{secStarGraphs}

In this section, we consider the example of the \emph{star graph} $\mathcal G$, i.e. a metric graph with $K \in \mathbb N$ infinite edges $(e_k)_{k = 1, \dots, K}$ identified to $(0, \infty)$ and connected at a common vertex $v$ identified with $0$. Let $\gamma \in \R$ be fixed. We define the Hamiltonian operator $\operatorname{H}_\delta$ with \emph{$\delta$-vertex conditions}, i.e. defined on the domain
\[
    D( \operatorname{H}_\delta ) = \left\{ \psi \in H^2( \mathcal G ): \psi_{e_1} (0) = \ldots = \psi_{e_K} (0) \text{ and } \sum_{k = 1}^K \psi_{e_k}' (0) = \gamma \psi (0) \right\},
\]
where $\psi(0)$ denotes the common value of $\psi$ at the vertex $v$. Observe that, for $\psi \in D( \operatorname{H}_\delta )$, the $\delta$-vertex condition can be written as
\begin{equation*} 
    A \Psi(0) + B \Psi'(0) = 0
\end{equation*}
with $A, B$ are the $(K \times K)$-matrices given by
\begin{equation} \label{eqDeltaMatrix}
    A =
    \begin{pmatrix}
        1 & -1 & 0 & \cdots & \cdots & 0 \\
        0 & 1 & -1 & \ddots & & \vdots \\
        \vdots & \ddots & \ddots & \ddots & \ddots & \vdots \\
        \vdots & & \ddots & \ddots & \ddots & 0 \\
        0 & & & \ddots & 1 & -1 \\
        -\gamma & 0 & \cdots & \cdots & 0 & 0
    \end{pmatrix}, \quad
    B =
    \begin{pmatrix}
        0 & 0 & \cdots & 0 \\
        \vdots & \vdots & & \vdots \\
        0 & 0 & \cdots & 0 \\
        1 & 1 & \cdots & 1
    \end{pmatrix}.
\end{equation}
This operator is self-adjoint (see \cite[Chapter $1$]{BeKu13}). Its quadratic form $Q_{\operatorname{H}_\delta}$ is given by
\[
    Q_{\operatorname{H}_\delta} (\psi) = \| \psi' \|_{L^2( \mathcal G )}^2 + \gamma |\psi (0)|^2
\]
and is defined on the energy domain
\[
    D( Q_{\operatorname{H}_\delta} ) = \left\{ \psi \in H^1( \mathcal G ) : \psi_{e_1} (0) = \ldots = \psi_{e_K} (0) \right\}.
\]
We also define the Hamiltonian operator $\operatorname{H}_{\delta'}$ with \emph{$\delta'$-vertex conditions} defined on the domain
\[
    D(\operatorname{H}_{\delta'}) = \left\{\psi \in H^2(\mathcal G): \psi_{e_1}'(0) = \ldots = \psi_{e_K}'(0) \text{ and } \sum_{k = 1}^K \psi_{e_k}(0) = \gamma \psi'(0) \right\},
\]
where $\psi'(0)$ denotes the common value of $\psi'$ at the vertex. The $\delta'$-vertex condition can be written in the form
\[
    B\Psi(0) + A\Psi'(0) = 0
\]
where $A$ and $B$ are given in \eqref{eqDeltaMatrix}. Its quadratic form $Q(\operatorname{H}_{\delta'})$ is given by
\[
    Q_{\operatorname{H}_{\delta'}}(\psi) = \| \psi' \|_{L^2( \mathcal G )}^2 + \frac{1}{\gamma} \sum_{k = 1}^K \left| \psi_{e_k} (0) \right|^2
\]
on the domain
\[
    D(Q_{\operatorname{H}_{\delta'}}) = H^1(\mathcal G).
\]
The operators $\operatorname{H}_\delta$ and $\operatorname{H}_{\delta'}$ are self-adjoint and their essential spectrum is given by $\sigma_{ess}(\operatorname{H}_\delta) = \sigma_{ess}(\operatorname{H}_{\delta'}) = [0, \infty)$ by Lemma \ref{prpEssSpec}. We proceed to compute explicitly the bottom of the spectrum of those operators in the following lemma. 
\begin{lem} \label{lemDiscSpecStarGraph}
    If $\gamma < 0$, the discrete spectrum of $\operatorname{H}_\delta$ and $\operatorname{H}_{\delta'}$ are given by
    \begin{equation}
        \label{eqSpecBott}
        \sigma_{dis} (\operatorname{H}_\delta) = \left\{ -\frac{\gamma^2}{K^2} \right\}, \quad \sigma_{dis} (\operatorname{H}_{\delta'}) = \left\{ -\frac{K^2}{\gamma^2} \right\}.
        \end{equation}
    In particular, the bottom of the spectrum is given by $l_{\operatorname{H}_\delta} = -\gamma^2 / K^2$ and $l_{\operatorname{H}_{\delta'}} = -K^2 / \gamma^2$. 
    
    If $\gamma \geq 0$, then $\sigma_{dis}(\H_\delta) = \sigma_{dis}(\H_{\delta'})=  \emptyset$.
\end{lem}
\begin{proof}
    We start by computing the discrete spectrum of $\operatorname{H}_\delta$. By self-adjointness of $\operatorname{H}_\delta$ and Lemma \ref{prpEssSpec}, we only have to consider the interval $(-\infty, 0)$ to seek eigenvalues. Let $\lambda \in (-\infty, 0)$. A solution $\phi \in H^2( \mathcal G )$ of the equation 
    \begin{equation} \label{eqLinEigODE}
        - \partial_{xx} \phi_{e_k} - \lambda \phi_{e_k} = 0
    \end{equation}
    for all $k = 1, \ldots, K$ is of the form
    \[
        \phi_{e_k} (x) =e^{- \sqrt{| \lambda |}x}, \quad x \in ( 0, \infty )
    \]
    for some $C_k \in \mathbb R$. However, such a $\phi$ satisfies the $\delta$-vertex conditions if and only if $\lambda = -\gamma^2 / K^2$ and $C_1 = \ldots = C_K$ for all $k = 1, \ldots, K$. Thus, the only negative eigenvalue of $\operatorname{H}_\delta$ is given by
    \[
        \lambda = - \frac{ \gamma^2 }{ K^2 }
    \]
    and is associated with the eigenvector
    \[
        \phi_{e_k} (x) = e^{\frac{ \gamma }{ N }x}, \quad x \in (0, \infty).
    \]
    Therefore, $-\gamma^2 / K^2$ is an isolated eigenvalue of multiplicity 1 of $\operatorname{H}_\delta$ and we have
    \[
        \sigma_{dis} ( \operatorname{H}_\delta ) = \left\{ - \frac{ \gamma^2 }{ K^2 } \right\}.
    \]
    The computation of the discrete spectrum of $\operatorname{H}_{\delta'}$ follows the same way, with the exception of the eigenvalue being $\lambda = -K^2 / \gamma^2$. This concludes the proof.
\end{proof}

Notice that by Theorem \ref{thmCriterion}, for $\gamma \geq 0$, a ground state does not exist as $\operatorname{H}_\delta$ and $\operatorname{H}_{\delta'}$ are positive operators. Thus, from now until the end of this section, we will consider only the case $\gamma < 0$. 

We now compute the stationary state of \eqref{eqStationaryNLS} for $\operatorname{H}_{\delta}$ and $\operatorname{H}_{\delta'}$.

\begin{proposition} \label{propStarGraphGS}
    Let $i \in \{ \delta, \delta'\}$ be fixed. The unique non-trivial solution to the equation
    \begin{equation} \label{eqStationaryStarGraph}
        \operatorname{H}_i \phi + \omega \phi + |\phi|^{p-1} \phi = 0
    \end{equation}
    up to a phase shift is given by $\phi_\omega = \phi_\omega(i) \in D( \operatorname{H}_i )$ such that:
    \begin{enumerate}
        \item for $\omega \neq 0$ and $p > 1$,
        \begin{equation} \label{eqStarGraphGS1}
            {\phi_\omega}_{e_k} (x) = \left( \frac{(p+1)\omega}{2} \right)^\frac{1}{p-1} \sinh{ \left( \frac{(p-1)\sqrt{\omega}}{2} x + b_\omega \right) }^{-\frac{2}{p-1}}
        \end{equation}
        with $x \in (0, \infty)$, $k = 1, \ldots, K$ and
        \[
            b_\omega = b_\omega(i) = \operatorname{arctanh} \left( \sqrt{-\frac{\omega}{l_{\operatorname{H}_i}} } \right);
        \]
        \item for $\omega = 0$ and $1 < p < 5$,
        \begin{equation} \label{eqStarGraphGS2}
            {\phi_0}_{e_k} (x) = \left( \frac{1-p}{\sqrt{2(p+1)}} x + b_0 \right)^{-\frac{2}{p-1}}
        \end{equation}
        with $x \in (0, \infty)$, $k = 1, \ldots, K$ and
        \[
            b_0 = b_0(p, i) = \left( - \frac{2}{l_{\operatorname{H}_i} (1 + p)} \right)^\frac{p-1}{2(p+1)}.
        \]
    \end{enumerate}
    For $\omega = 0$ and $p \geq 5$, there is no non-trivial $H^1(\G)$ solution to \eqref{eqStationaryStarGraph}.
\end{proposition}
\begin{proof}
    Let $p > 1$, $\omega \in (0, -l_{\operatorname{H}_i})$ and $\gamma < 0$ be fixed. For $i = \delta$, it follows from \cite[Theorem $1$]{KaOh09}, \cite[Theorem $4.1.$]{PaGo18} that the unique solution to \eqref{eqStationaryStarGraph} is given by \eqref{eqStarGraphGS1} up to a phase shift. We then proceed to compute the stationary state for $i = \delta'$. The equation \eqref{eqStationaryStarGraph} can be written on every edge as
    \[
        -\phi_{e_k}'' + \omega \phi_{e_k} + |\phi_{e_k}|^{p-1} \phi_{e_k} = 0,
    \]
    so that, for $k = 1, \ldots, K$, we have
    \[
        \phi_{e_k}(x) = \sigma_k \left( \frac{(p+1)\omega}{2} \right)^{1/(p-1)} \sinh{ \left( \frac{(p-1)\sqrt{\omega}}{2} x + b_k \right) }^{-2/(p-1)}
    \]
    where $|\sigma_k| = 1$ and $a_k > 0$ is a translation parameter. By continuity of $\phi'$ at the vertex, we immediately have $\sigma_1 = \ldots = \sigma_K$ and
    \[
        \cosh(b_1) \sinh(b_1)^{-\frac{2}{p-1} - 1} = \ldots = \cosh(b_K) \sinh(b_K)^{-\frac{2}{p-1} - 1},
    \]
    so by monotonicity, it follows that $b_1 = \ldots = b_K = b$. Furthermore, the jump of the function imposes that
    \[
        K \sinh(b)^{-\frac{2}{p-1}} = -\gamma \sqrt{\omega} \cosh(b) \sinh(b)^{-\frac{2}{p-1} - 1},
    \]
    so by \eqref{eqSpecBott} we get
    \begin{equation}\label{eqBOmega}
         b = b_\omega(i) = \operatorname{arctanh} \left( \sqrt{-\frac{\omega}{l_{\operatorname{H}_i}}} \right).
    \end{equation}
    We now turn to the case $\omega = 0$. By direct computation, one can verify that a solution of the form \eqref{eqStarGraphGS2} solves \eqref{eqStationaryNLS} on the half-line $\R^+$. Moreover, by classical theory, this solution is the unique positive real-valued one. The shift parameter $b_0$ can then be determined by the same procedure as above. Finally, it is straightforward to check that the resulting profile belongs to $H^1(\R^+)$ for $p \in (1,5)$, whereas for $p > 5$ it fails to be in $L^2(\R^+)$. This completes the proof.
\end{proof}
Proposition \ref{propStarGraphGS} is illustrated in Figure \ref{figStarGraphGS}. We used the Python library Grafidi presented in \cite{BeDuLeC22} to represent the ground state on a star graph. 

\begin{figure}[H] 
    \centering
    \includegraphics[height=0.21\textheight]{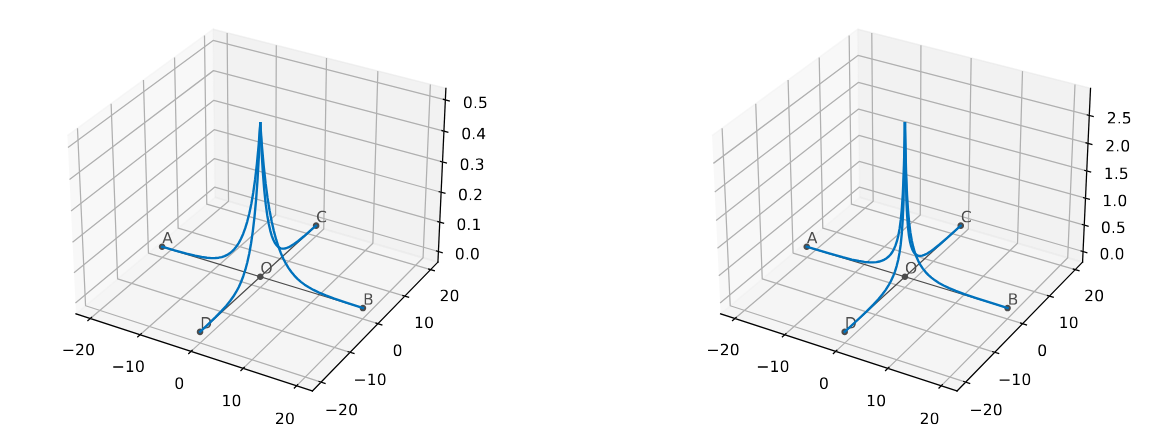}
    \caption{Action ground state for $\delta$-type condition (left) and $\delta'$-type condition (right) for $p=3$, $K = 4$, $\gamma = -2$, $\omega = 1/8$.}
   \label{figStarGraphGS}
\end{figure}

We immediately deduce the following corollary from Lemma \ref{lemDiscSpecStarGraph} and Theorem \ref{thmCriterion}.

\begin{corollary}
    The following statements hold:
    \begin{enumerate}
        \item Let $p > 1$, $\omega \in (0, -l_{\operatorname{H}_i})$, $\gamma < 0$ and $i \in \{ \delta, \delta'\}$ be fixed. The set of action ground states on the star graph $\mathcal G$ is given by $\mathcal A_\omega = \{ e^{i \theta} \phi_\omega : \theta \in [0, 2 \pi) \}$, where $\phi_\omega$ is given by \eqref{eqStarGraphGS1}, and is stable.
        \item Let $1 < p < 5$, $\omega = 0$, $\gamma < 0$ and $i \in \{ \delta, \delta'\}$ be fixed. The set of action ground states on the star graph $\mathcal G$ is given by $\mathcal A_0 = \{ e^{i \theta} \phi_0: \theta \in [0, 2 \pi) \}$, where $\phi_0$ is given by \eqref{eqStarGraphGS2}, and is stable.
    \end{enumerate}
    For $p \geq 5$ and $\omega = 0$, there are no action ground states.
\end{corollary}

Since the solution is unique up to phase shifts, we will henceforth refer to $\phi_\omega$, given by \eqref{eqStarGraphGS1}–\eqref{eqStarGraphGS2}, as the action ground state of $S_\omega$. We now investigate the relation between the frequency parameter $\omega$ and the mass of the corresponding action ground state for different values of $p$. The next proposition shows a sharp dichotomy: for $1 < p < 5$, the mass of the action ground state ranges within a finite interval, while for $p \geq 5$, it spans the entire positive half-line $\mathbb{R}^+$.

\begin{proposition}\label{prpSec4Main}
    Let $\gamma < 0$ and $i \in \{ \delta, \delta'\}$ be fixed. The following properties hold:
    \begin{enumerate}
        \item For $1 < p < 5$, there exists $\mu = \mu(p) < \infty$ such that the function $\omega \in [0, -l_{\operatorname{H}_i}) \mapsto M(\phi_\omega)$ takes values in $(0, \mu)$ and is strictly decreasing.
        \item For $p \geq 5$, the function $\omega \in (0, -l_{\operatorname{H}_i}) \mapsto M(\phi_\omega)$ takes values in $(0, \infty)$ and is strictly decreasing.
    \end{enumerate}
\end{proposition}
\begin{proof}
    Let be $i \in \{ \delta, \delta'\}$ be fixed. We first consider $p \in (1, \infty) \setminus \{ 5 \}$. Applying the change of variable $y = (p-1)\sqrt{\omega} x / 2 + b_\omega$, we have
    \begin{equation*} \label{eqMassGS}
        M(\phi_\omega) = K \left( \frac{p+1}{2} \right)^{\frac{2}{p-1}} \omega^{\frac{5-p}{2(p-1)}} \int_{b_\omega}^\infty \sinh (y)^{-\frac{4}{p-1}} dy,
    \end{equation*}
    with
    \begin{equation*}
        b_\omega \to
        \begin{cases}
            0 \text{ as } \omega \to 0, \\
            \infty \text{ as } \omega \to -l_{\operatorname{H}_i}.
        \end{cases}
    \end{equation*}
    We immediately have
    \begin{equation} \label{eqMassLim}
        M(\phi_\omega) \to 0 \text{ as } \omega \to -l_{\operatorname{H}_i}
    \end{equation}
    by the dominated convergence theorem.
    
    We consider now the case where $\omega$ is asymptotically close to $0$. For $1 < p < 5$, we have that
    \begin{equation} \label{eqMassOmegaNull}
        M(\phi_0) = \frac{K\sqrt{2(1+p)}}{p-1} b_0^\frac{p-5}{p-1}.
    \end{equation}
    As $\omega \in [0,-l_{\operatorname{H}}) \mapsto M(\phi_\omega)$ is injective and continuous by Remark \ref{rkMassCont}, we obtain that it is strictly decreasing and taking values in $(0, M(\phi_0))$ by \eqref{eqMassLim}-\eqref{eqMassOmegaNull}.
    
    For $p > 5$, we separate the integral in two terms: an integral between $b_\omega$ and $b_\omega + 1$, and an integral between $b_\omega + 1$ and $\infty$. Observe that
    \[
        \omega^\frac{5-p}{2(p-1)} \int_{b_\omega + 1}^\infty \sinh(y)^{-\frac{4}{p-1}} dy \to 0 \text{ as } \omega \to 0.
    \]
    Furthermore, by Taylor expansion of $\sinh$, we have
    \[
        \int_{b_\omega}^{b_\omega + 1} (2 \sinh (y))^{-\frac{4}{p-1}} dy
        = \int_{b_\omega}^{b_\omega + 1} y^{-\frac{4}{p-1}} (1 + \epsilon(y))^{-\frac{4}{p-1}} dy
    \]
    where $\epsilon(y) = o_{y \to 0}(1)$. Furthermore, we have
    \begin{equation} \label{eqMassConvOrder}
        \int_{b_\omega}^{b_\omega + 1} y^{-\frac{4}{p-1}} dy = C \left( (b_\omega + 1)^\frac{p-5}{p-1} - b_\omega^\frac{p-5}{p-4} \right),
    \end{equation}
    where $C > 0$. We get
    \[
        \omega^{\frac{5-p}{2(p-1)}} (b_\omega + 1)^\frac{p-5}{p-1} \to \infty \text{ as } \omega \to 0
    \]
    so
    \begin{equation} \label{eqMassLimSupCrit}
        M(\phi_\omega) \to \infty \text{ as } \omega \to 0.
    \end{equation}
    Using Remark \ref{rkMassCont}, we obtain that it is strictly decreasing and taking values in $(0, \infty)$ by \eqref{eqMassLim}-\eqref{eqMassLimSupCrit}.
    
    For $p = 5$, we have
    \begin{equation} \label{eqPEq5}
        \begin{aligned}
             M(\phi_\omega) & = K \sqrt{3} \int_{b_\omega}^\infty \sinh (y)^{-1} dy  = K \sqrt{3} \left( \ln(e^{b_\omega} + 1) - \ln(e^{b_\omega} - 1) \right)  \\ 
             &= K \sqrt{3} \ln \left( \sqrt{-\frac{l_{\operatorname{H}_i}}{\omega}} \right).
        \end{aligned}
    \end{equation}
    We deduce from this explicit formulation all the properties mentioned in the statement of the proposition, and this concludes the proof.
\end{proof}

\begin{remark}
     By \eqref{eqStarGraphGS1}, for $p = 3$, we get on each edge $e \in \mathcal{E}$
    \begin{align*}
        \| \phi_{e} \|_{L^2{(\R^+)}}^2 = 2\omega \int_0^\infty \sinh\left( \sqrt{\omega} x + b_\omega \right)^{-2} dx = 2\sqrt{\omega} \left(\frac{1}{\tanh(b_\omega)}  - 1 \right)
    \end{align*}
    By \eqref{eqBOmega}, we obtain
    \begin{equation*}
         \| \phi_{e} \|_{L^2{(\R^+)}}^2 = 2 \left( \sqrt{-l_{\H_i}} - \sqrt{\omega} \right). 
    \end{equation*}
    Consequently, for $|\mathcal{E}| = K$, we obtain 
    \begin{equation*}
        \| \phi_\omega \|_{L^2(\G)}^2 = K  \| \phi_{e} \|_{L^2{(\R^+)}}^2 = 2K \left(\frac{|\gamma|}{K} - \sqrt{\omega}\right)
    \end{equation*}
    for the $\delta$-vertex conditions and, in the same way, 
    \begin{equation*}
           \| \phi_\omega \|_{L^2(\G)}^2 = 2K \left(\frac{K}{|\gamma|} - \sqrt{\omega}\right)
    \end{equation*}
    for the $\delta'$-vertex condition. Thus, for the $\delta$-vertex conditions, the maximum mass allowed is $2|\gamma|$, and is independent of the number of edges. In contrast, for the $\delta'$-vertex conditions, the maximum mass is given by $2K^2/|\gamma|$, which increases as the number of edges grows. 
    
    Furthermore, we explicitly computed the masses for $p = 5$ for both $\delta$ and $\delta'$-vertex conditions in \eqref{eqPEq5}. Figure \ref{fig2} illustrates the different behaviors for $p=3, 5, 7$, where $p = 7$ is computed numerically. 
\end{remark}

\begin{figure}[H]
    \centering
    \begin{minipage}{0.45\textwidth}
        \centering
        \includegraphics[height=0.25\textheight]{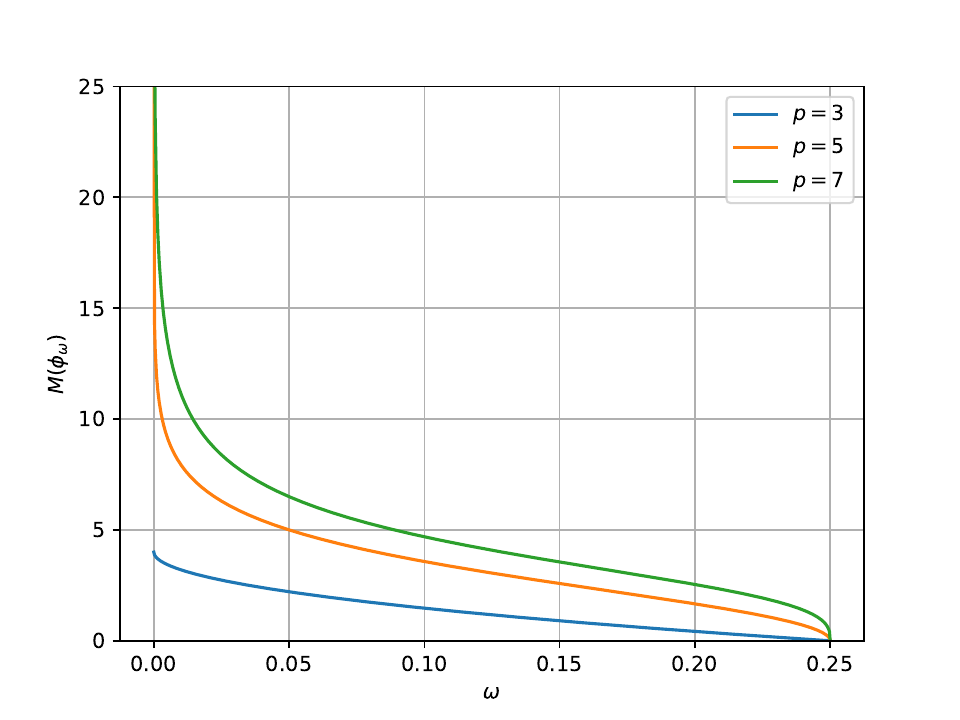}
    \end{minipage}%
    \begin{minipage}{0.45\textwidth}
        \centering
        \includegraphics[height=0.25\textheight]{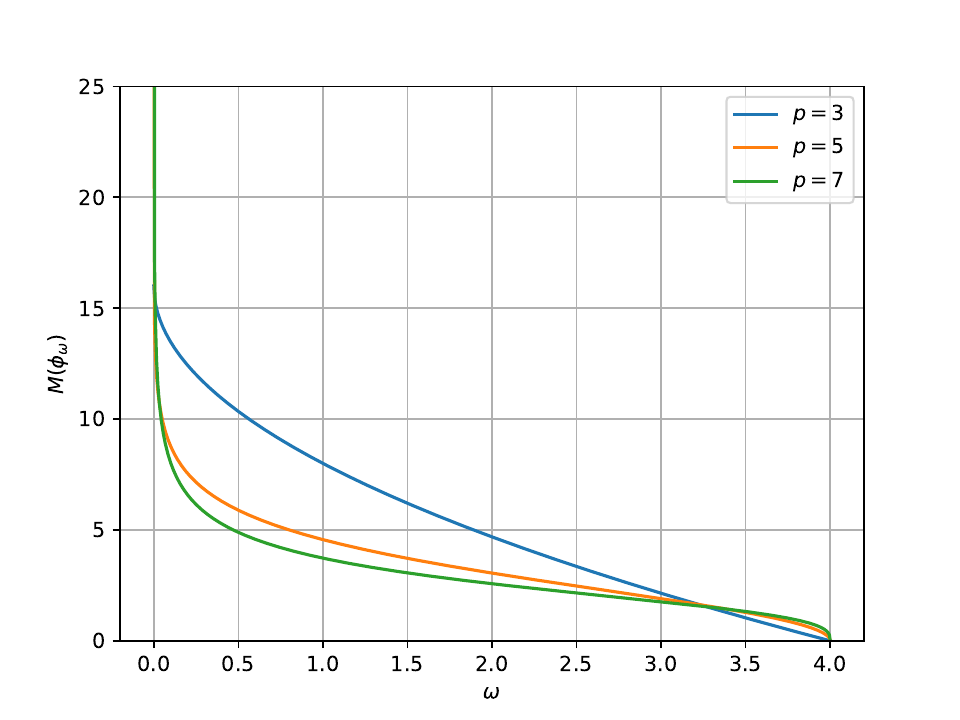}
    \end{minipage}
    \caption{Mass of the ground state for $\delta$-type condition (left) and $\delta'$-type condition (right) for $p=3,5,7$, $K = 4$, $\gamma = -2$.}
    \label{fig2}
\end{figure}

\bibliographystyle{abbrv}
\bibliography{biblio}

\end{document}